\documentclass[11pt]{amsart}
\usepackage{amsmath, amssymb, amsthm}
\usepackage{enumerate}
\usepackage[margin=1in]{geometry}
\usepackage{esint}
\usepackage{xcolor}
\usepackage{hyperref}
\hypersetup{colorlinks = true, linkbordercolor = red}

\usepackage[pdftex]{graphicx}
\pdfoutput=1

\newtheorem{theorem}{Theorem}[section]
\newtheorem{prop}[theorem]{Proposition}
\newtheorem{lemma}[theorem]{Lemma}

\newtheorem{defn}[theorem]{Definition}
\newtheorem{rem}[theorem]{Remark}

\numberwithin{equation}{section}

\def\R{\mathbb{R}}
\def\C{\mathbb{C}}

\def\N{\mathbb{N}}

\def\e{\varepsilon}
\def\eps{\epsilon}

\def\n{\nu}

\def\D{\Delta}
\def\O{\Omega}
\def\L{\Lambda}
\def\a{\alpha}

\def\g{\gamma}
\def\d{\delta}

\def\f{\varphi}
\def\k{\kappa}
\def\l{\lambda}
\def\z{\zeta}

\def\H{\mathcal{H}}

\def\Dd{\mathbb{D}}
\def\X{\mathcal{X}}

\def\n{\nabla}

\def\ind{\mathrm{ind}}
\def\dim{\mathrm{dim}}

\def\supp{\mathrm{supp}\,}

\newcommand{\chrc}[1]{\mathrm{1}_{#1}}
\renewcommand{\div}{\text{div}}

\def\de{\partial}

\def\XXint#1#2#3{{\setbox0=\hbox{$#1{#2#3}{\int}$ }
\vcenter{\hbox{$#2#3$ }}\kern-.6\wd0}}

\title[Free boundaries of finite index]{One-phase free boundary solutions of finite Morse index}

\author{Jos\'e Basulto}
\address{Jos\'e Basulto, Facultad de Matem\'aticas, Pontificia Universidad Cat\'olica de Chile, Avenida Vicu\~na Mackenna 4860, Santiago 7820436, Chile}
\email{jos.basulto@uc.cl}

\author{Nikola Kamburov}
\address{Nikola Kamburov, Facultad de Matem\'aticas, Pontificia Universidad Cat\'olica de Chile, Avenida Vicu\~na Mackenna 4860, Santiago 7820436, Chile}
\email{nikamburov@uc.cl}

\thanks{The authors were partially supported by ANID Fondecyt Regular No.\ 1201087.}

\begin{document}

\begin{abstract}

We study global solutions to the classical one-phase free boundary problem that have finite Morse index relative to the Alt-Caffarelli functional.  We show that such solutions are stable outside a compact set and characterize the index as the  maximal number of linearly independent $L^2$ integrable eigenfunctions of the corresponding Robin eigenvalue problem, associated to negative eigenvalues. As an application, we obtain a complete classification of global solutions of finite Morse index in the plane. Our results are counterparts to the minimal surface theorems of Fischer-Colbrie and Gulliver.
\end{abstract}


\bibliographystyle{alpha}
\maketitle

\section{Introduction}

The investigation into the solutions of the classical one-phase free boundary problem (FBP) in a bounded domain $D\subset \R^n$:
  \begin{equation}\label{FBP}
\left\{
\begin{aligned}
 u\geq 0 & \quad \mbox{in } D, \\
   \Delta  u=0 & \quad \mbox{in } D^+(u):=\{x \in D: u(x)>0\}, \\
   |\nabla u|=1 & \quad \mbox{on }  F(u):=\de D^+(u) \cap \O,
  \end{aligned}\right.
  \end{equation}
with fixed boundary conditions on $\de D$, was spurred by the highly influential 1981 paper \cite{AC} of Alt and Caffarelli. 
Motivated by interface models in fluid mechanics and materials science (\cite{FriedBook}), the authors formulated the energy functional
\begin{equation}\label{eq:OPfunc}
I(v, D):=\int_D \left(|\n v|^2 + \chrc{(0,\infty)}(v)\right) dx, \qquad v:D\to [0,\infty),
\end{equation}
and showed that its minimizers $u$ satisfy \eqref{FBP} in a suitable weak sense, with the \emph{free boundary} $F(u)$ being smooth except on a relatively closed subset of zero $(n-1)$-dimensional Hausdorff measure. The methods behind this regularity result were inspired by earlier developments in minimal surface theory and the paper revealed fascinating analogies between minimal surfaces and free boundaries. In particular, the regularity theory for both objects is closely linked to the classification of \emph{globally defined} critical points of the corresponding energy functional. 

The focus of this article is on global classical solutions $u:\R^n \to [0,\infty)$ of the one-phase FBP that are \emph{not necessarily energy minimizing} (see Definition \ref{def:class}). Such higher-order critical points are of interest in fluid mechanics (\cite{BSS76, TraizetHV}) and in electrostatics (\cite{sirakov2002}). 
One way to delimit the much wider study of general global solutions to \eqref{FBP} is to impose a \emph{topological} restriction of the positive phase $\O:=\{u>0\}$. In $n=2$ dimensions, the topological classification problem was first raised by Hauswirth, H\'elein and Pacard \cite{HHP}, who conjectured that, up to similarity transformations, the list of solutions to \eqref{FBP} in $\R^2$, for which $\O$ is connected and has \emph{finite topology}, that is, for which $F(u)=\de \O$ consists of finitely many components, consists of
\begin{itemize}
\item the \textit{one-plane solution} $P(x_1,x_2):=x_1^+$,
\item the \textit{disk-complement solution} $L(x_1,x_2):=(\log|(x_1,x_2)|)^+$,
\item and the solution $H(x_1, x_2)$, discovered by the authors in the same paper \cite{HHP}. We refer to it as the \textit{double hairpin solution}, given that its positive phase $\{H>0\}$ is a region trapped between the two hairpin-like catenary curves: $$\{(x_1,x_2)\in \R^2: x_2=\pm (\pi/2 + \cosh x_1)\}.$$
\end{itemize}
Khavinson, Lundberg and Teodorescu \cite{KLT13} confirmed that $P$ and $H$ are the only solutions with simply connected positive phase in $\R^2$. The full conjecture was resolved in the affirmative by Traizet \cite{Traizet} who established a remarkable correspondence between solutions of \eqref{FBP} in the plane and minimal surfaces in $\R^3$. The list is certainly not exhaustive of all the possible global solutions in $\R^2$: there exists a family of periodic solutions with infinite topology (\cite{BSS76, JK19}). In dimension $n\geq 3$, in addition to the solutions above, extended in constant fashion in the remaining $(n-2)$ directions, one expects a zoo of global solutions of finite topology, starting from the axially symmetric solution constructed by Liu, Wang and Wei \cite{LWW21}. 

A second natural approach is to instead consider a variation-theoretic constraint on the space of global solutions. For a classical solution $u$ of \eqref{FBP} in a domain $D\subseteq \R^n$, the \emph{second variation form} of the Alt-Caffarelli energy $I(\cdot, D)$ at $u$ is given by (\cite{CJK})
\begin{equation}\label{eq:secvar}
\begin{aligned}
Q(\phi, \phi) & :=\int_{D^+(u)} |\n \phi|^2 \, dx - \int_{F(u)} H \phi^2 \, d\H^{n-1}, \quad \text{for } \phi \in C^\infty_c(D^+(u)\cup F(u)), 
\end{aligned}
\end{equation}
where $H$ denotes the mean curvature of $F(u)$ with respect to the outer unit normal $\nu = - \n u$. The \emph{Morse index} of $u$ in $D$ with respect to the Alt-Caffarelli energy $I(\cdot, D)$ is then defined to be the index of the quadratic form $Q$, that is
\begin{equation}\label{eq:indexdef}
\begin{split}
\ind(u, D^+(u)):=\sup\{ \text{dim } V: Q(\cdot,\cdot) &\text{ is negative definite} \\ & \text{ on the subspace } V \subseteq C^\infty_c(D^+(u)\cup F(u))\},
\end{split}
\end{equation}
with $\ind(u,D^+(u))=0$ meaning that $u$ is \emph{stable} in $D$. In this paper, we will be interested in the global solutions $u:\R^n\to [0,\infty)$ of the one-phase FBP that possess \emph{finite index} 
$$
\ind(u):=\ind(u,[\R^n]^+(u)) <\infty.
$$
Our main result is that in the case of the plane $\R^2$, solutions of finite index also have finite topology. As a result, we can characterize all global solutions in $\R^2$ of finite index according to its value. 

\begin{theorem}\label{thm:class} Let $u$ be a global classical solution of the one-phase FBP \eqref{FBP} in $\R^2$ that has connected positive phase $\O:=\{u>0\}$ and finite index. Then $u$ has finite topology. Thus, it is either the one-plane solution $u=P$, the disk-complement solution $u=L$, or the double hairpin solution $u=H$, up to a similarity transformation. In the first instance, the solution is stable, while in the remaining two, the index of $u$ is $1$. 
\end{theorem}

The theorem above is the free boundary analogue of the celebrated result of Fischer-Colbrie \cite{FCminimalFM} and Gulliver \cite{Gul1986} that complete minimal surfaces of finite Morse index in $\R^3$ actually have finite topology. In the exact same spirit, the intermediate step to establishing Theorem \ref{thm:class} is proving that solutions of \eqref{FBP} of finite index enjoy a \emph{finite total mean curvature bound}
\begin{equation}\label{eq:fincurvintro}
\int_{F(u)} |H|\, d\H^1 <\infty,
\end{equation}
which, in turn, entails the finite connectivity of $F(u) = \de\O$. The computation of the precise value of the index for each of the nontrivial solutions $u=L$ and $u=H$ is then carried out using the conformal equivalence of $\overline{\O}$ to the unit disk $\overline{\Dd}$ minus a finite number of points -- in another nod to the minimal surface story. 

The curvature bound \eqref{eq:fincurvintro} follows from the fact that (in any dimension) a global solution of finite index must be stable outside a compact set (Proposition \ref{prop:stab}). In order to prove this statement in our free boundary setting, we require some basic spectral results for the associated local Dirichlet-Robin eigenvalue problem
  \begin{equation}\label{eq:eig0}
\left\{
\begin{aligned}
   -\Delta  \f = \l \f & \quad \mbox{in } D^+(u), \\
   \f_\nu - H\f = 0 & \quad \mbox{on }  F(u), \\
	\f= 0 & \quad \mbox{on } \de D^+(u) \setminus F(u),
  \end{aligned}\right.
  \end{equation}
where the positive phase $D^+(u)\subset \R^n$ is bounded and connected. Since we weren't able to find in the literature explicit treatments of this type of eigenvalue problem, we build its spectral theory here, based on Sobolev trace inequalities that are specific for the one-phase FBP context (see Proposition \ref{prop:trace}). We furthermore prove that a global solution $u$ has finite index $k=\ind(u)$ if and only if there are exactly $k$ linearly independent $L^2(\{u>0\})$ eigenfunctions of \eqref{eq:eig0}, associated to negative eigenvalues, such that $Q$ is positive definite on the $L^2$-orthogonal complement of the space generated by them (Proposition \ref{prop:L2}).

Results relating the index of a minimal surface to its topology are abundant in the minimal surface literature (see \cite{FCminimalFM, Gul1986, Ros2006, ChodMaxI, ChodMaxII} and references therein). The same question has also been explored in the context of finite-ended solutions to the Allen-Cahn equation in the plane (\cite{KowLiuPacard2012, PKP2013, GuiLiuWei2016, WangWei2019}) as well as for its free boundary analogue (\cite{wang2015structure, DuGuiWang2022}). For the one-phase FBP itself, there has been a recent surge of interest into energy non-minimizing solutions: \cite{JK16, JK19} (under a topological constraint),  \cite{FR-RO2019, KW23} (under a stability condition) and \cite{engelstein2022graphical} (under a graphical free boundary condition). The goal of the present article is contribute to these latest developments.

The paper is organized as follows. In Section \ref{sec:prelim} we give the necessary preliminaries as well as the spectral theory of the eigenvalue problem \eqref{eq:eig0} on bounded domains. In Section \ref{sec:globalsol} we present the main results characterizing global solutions of \eqref{FBP} of finite index (Propositions \ref{prop:stab} and \ref{prop:L2}). In Section \ref{sec:dimtwo}, we focus on the two-dimensional problem, where we prove the finite curvature bound \eqref{eq:fincurvintro} and demonstrate how it implies the finite topology property (Proposition \ref{prop:fincurvtop}). Lastly, in Section \ref{sec:Morsecomp}, we carry out the computation of the Morse indices of the three global solutions of finite topology,  $P$, $L$ and $H$, and complete the proof of Theorem \ref{thm:class}. In the Appendix, we provide the Sobolev trace inequalities underlying the spectral theory of the problem \eqref{eq:eig0}.

\section{Preliminaries}\label{sec:prelim}

We start with the definition of a classical solution of the one-phase free boundary problem.

\begin{defn}\label{def:class} A classical solution $u$ to the one-phase FBP in a domain $D\subseteq \R^n$ is a non-negative function $u\in C(\overline{D})$, whose free boundary $F(u)$ is a smooth, oriented hypersurface in $D$, with $u>0$ on one side and $u=0$, on the other. Moreover, $u\in C^\infty(\overline{D^+(u)})$ is harmonic in $D^+(u)$ and satisfies the gradient condition $|\n u|=1$ on $F(u)$ in a pointwise sense from the positive side.
\end{defn}
Note that a classical solution $u$ is Lipschitz continuous in $D$ and the mean curvature $H$ of the free boundary $F(u)$ is a smooth function $H\in C^\infty(\overline{F(u)})$.

Denote by $\mathcal{X}_{D^+(u)}:=C^\infty_c(\overline{D^+(u)}\setminus \de D)$ the space of functions on $D^+(u)$ that are smooth up to the free boundary $F(u)$ but are compactly supported away from the fixed boundary $\de D$. Similarly, for a connected component $\O$ of $D^+(u)$, denote the space of smooth functions on it, supported away from the fixed boundary, by $\X_\O:=C^\infty_c(\overline{\O}\setminus \de D)$.  We may treat $\X_\O$ as a subspace of $\mathcal{X}_{D^+(u)}$ by defining each $\phi \in \X_\O$ to be $0$ in $D^+(u)\setminus \overline{\O}$. 

The second variation form of the Alt-Caffarelli functional $I(\cdot, D)$ at the classical solution $u$ of \eqref{FBP} in $D$ is given by the symmetric bilinear form
\[
Q(\phi, \psi; D^+(u)):= \int_{D^+(u)} \n \phi \cdot \n \psi \, dx - \int_{F(u)} H \phi \psi\, d\H^{n-1} \quad \text{for } \phi,\psi \in \mathcal{X}_{D^+(u)}.
\]
We will often restrict the form to act only on functions supported on a specific connected component $\O\subseteq D^+(u)$:
\[
Q(\phi, \psi; \O):= Q(\phi, \psi; D^+(u)) \quad \text{restricted to } \phi,\psi \in \mathcal{X}_{\O}\subseteq \mathcal{X}_{D^+(u)}.
\]

\begin{defn} The (Morse) index $\ind(u,D^+(u))$ of a classical solution $u$ of the one-phase FBP in a domain $D\subseteq \R^n$ is the index of the variation form $Q(\cdot, \cdot; D^+(u))$, i.e.\
\begin{equation*}
\begin{aligned}
\ind(u, D^+(u)) &:= \ind~ Q(\cdot, \cdot; D^+(u))  \\ 
& := \sup\{ \text{dim } V: Q(\cdot,\cdot; D^+(u)) \text{ is negative definite}  \text{ on the subspace } V \subseteq \X_{D^+(u)}\}.
\end{aligned}
\end{equation*}
It will also be convenient to define the index on a particular connected component $\O$ of $D^+(u)$: $$\ind (u, \O):= \ind~ Q(\cdot,\cdot;\O).$$
We will say that $u$ is \emph{stable} in $D$ (resp. \emph{stable} in $\O$) if $$\ind(u, D^+(u)) = 0 \quad \text{(resp. } \ind(u,\O) = 0).$$
\end{defn}

\begin{rem}[Monotonicity of the index with respect to domain inclusion] \label{rem:indmon}
Assume that $u$ is a classical solution of \eqref{FBP} in a bounded domain $D$ and that $D'\subseteq D$ is a subdomain. Let $\O'$ be a connected component $(D')^+(u)$ and $\O$ -- a connected component of $D^+(u)$, such that $\O'\subseteq \O$. Since we may consider $\X_{(D')^+(u)}\subseteq \X_{D^+(u)}$ and $\X_{\O'}\subseteq \X_{\O}$, we immediately deduce the monotonicity property of the index with respect to domain inclusions:
\begin{equation}\label{eq:indmon}
\ind(u,(D')^+(u))\leq \ind(u, D^+(u)) \quad \text{and} \quad \ind(u,\O') \leq \ind(u,\O). 
\end{equation}
\end{rem}

\begin{rem}
Because of the monotonicity property, one can obtain the index $$\ind(u):=\ind(u, [\R^n]^+(u))$$ of a global solution $u:\R^n \to [0,\infty)$ of \eqref{FBP} as a limit of local indices over bounded subdomains that exhaust $\R^n$. In particular, assume that the positive phase $\O:=[\R^n]^+(u)$ is connected and $0\in \O$, and let $\O_R$ be the connected component of $\O\subset B_R$ containing $0$. Then by \eqref{eq:indmon}, $\ind(u, \O_{R})$ is increasing in $R$ and $\ind(u, \O_{R}) \leq \ind(u)$ for all $R>0$, so that $\lim_{R\to\infty} \ind(u, \O_{R}) \leq \ind(u)$. Conversely, if $V\subset \X_{\O}$ is a finite dimensional subspace of $\X_\O$ such that $Q$ is negative definite on $V$, then any (finite) basis of functions for $V$ is contained in $\X_{\O_R}$ for a large enough $R$. Thus, $V\subseteq \X_{\O_R}$ which implies  $\lim_{R\to\infty} \ind(u, \O_{R}) \geq \ind(u)$. We conclude that
\begin{equation}\label{eq:indlim}
\ind(u)=\lim_{R\to\infty} \ind(u, \O_R). 
\end{equation}
The characterization \eqref{eq:indlim} of the global index is useful due to the \emph{finiteness} of the local index $\ind(u, \O_R)$ on each bounded and connected $\O_R$. This can be seen by relating the local index to the finite number of negative eigenvalues of the associated eigenvalue problem, which we will do next. 
\end{rem}

Let $\O$ be a bounded connected component of the positive phase of a classical solution $u$ to \eqref{FBP} in a bounded domain $D$. To the form $Q(\cdot,\cdot;\O)$ one naturally associates the eigenvalue problem
  \begin{equation}\label{eq:eig}
\left\{
\begin{aligned}
   -\Delta  \f = \l \f & \quad \mbox{in } \O, \\
   \f_\nu - H\phi = 0 & \quad \mbox{on }  F(u)\cap \de \O, \\
	\f= 0 & \quad \mbox{on } \de\O \setminus F(u).
  \end{aligned}\right.
  \end{equation}
At this stage it is useful to extend the domain $[\X_\O]^2$ of $Q(\cdot, \cdot; \O)$ to the Sobolev space $[H^1_F(\O)]^2$, where
$$
H^1_F(\O):=\text{the closure of } C^\infty_c(\Omega \cup (F(u)\cap \de\O)) \text{ in } H^1(\O).
$$
We have
\[
Q(\phi, \psi; \O):= \int_{\O} \n \phi \cdot \n \psi \, dx - \int_{F(u)\cap \de \O} H (T\phi) (T\psi)\, d\H^{n-1} \quad \text{for } \phi,\psi \in H^1_F(\O),
\]
where $T: H^1_F(\O) \to L^2(F(u)\cap \de \O)$ denotes the trace operator onto the free boundary portion of $\de \O$. Clearly, the index of $Q(\cdot, \cdot; \O)$ doesn't change when $\X_\O$ is replaced with $H^1_F(\O),$ and the next theorem equates it to the number of negative eigenvalues of \eqref{eq:eig} (counted with multiplicity).

\begin{theorem}\label{thm:std} Let $u$ be a classical solution of \eqref{FBP} in a bounded domain $D\subset \R^n$ and let $\O$ be a connected component of $D^+(u)$. Then there exist a nondecreasing sequence of eigenvalues $\{\l_k\}_{k\in \N}\subset \R$, with \mbox{$\lim_{k\to\infty} \l_k = +\infty$,} and a set of corresponding eigenfunctions $\{\f_k\}_{k\in \N}\subset H^1_F(\O)$ satisfying \eqref{eq:eig} in a weak sense:
\begin{equation}\label{thm:std:weak}
Q(\f_k, \psi;\O) = \l_k\langle \f_k, \psi\rangle_{L^2(\O)} \quad \text{for all } \psi \in H^1_F(\O),
\end{equation}
such that $\{\f_k\}_{k\in \N}$ is an orthonormal basis for $L^2(\O)$. The eigenfunctions $\{\f_k\}_{k\in\N}$ are, in fact, smooth in $\O$ up to the free boundary portion $F(u)\cap \de \O$ and satisfy the Robin boundary condition there classically. Moreover, the Courant-Fischer minimax characterization of the eigenvalues holds:
\begin{align}
\l_1(\O) &= \min\{ Q(\psi,\psi;\O): \psi \in H^1_F(\O) \text{ with } \|\psi\|_{L^2(\O)}=1\}, \label{thm:std:first} \\
\l_k(\O) &= \min_{V_k\in S_k} \max\{ Q(\psi,\psi;\O): \psi \in V_k \text{ with } \|\psi\|_{L^2(\O)}=1\} \quad \text{for } k\geq 2, \label{thm:std:Courant-Fischer}
\end{align}
where $S_k$ is the set of $k$-dimensional subspaces of $H^1_F(\O)$. In particular, $$\ind(u,\O)=\max\{k\in \N: \l_k<0\}.$$ 
\end{theorem}
\begin{proof}
Denote $F:=F(u)\cap \de \O$.
The proof uses standard arguments from the theory of elliptic mixed boundary value problems (for example, see \cite[Chapter 1]{DamascelliPacellaBook}). We will briefly mention some of the key details for the reader's benefit. Define for $\L>0$, the bilinear symmetric form $Q_\L: H^1_F(\O)\times H^1_F(\O) \to \R$ 
\[
Q_\L(\f,\psi) : =Q(\f, \psi;\O) + \L \langle \f, \psi\rangle_{L^2(\O)}.
\]
Since $H\in L^\infty(F(u)\cap \de\O)$ and $u$ is Lipschitz continuous, the $L^2$-trace inequality \eqref{prop:trace:ineq1} yields that $Q_\L$ is continuous in $[H^1_F(\O)]^2$. Furthermore, for $\L\geq \L_0$ large enough, $Q_\L$ is coercive in $H^1_F(\O)$. Indeed, by \eqref{prop:trace:ineq2} we have 
\[
\|T\f\|_{L^2(F)}^2 \leq \e \|\n \f\|^2_{L^2(\O)} + C/\e \|\f\|^2_{L^2(\O)}
\]
for all $\e>0$ and some constant $C$ depending on $n$, and the Lipschitz constant of $u$. As a result,
\begin{align*}
Q_\L(\f,\f) &\geq \|\n \f\|^2_{L^2(\O)} + \L \|\f\|^2_{L^2(\O)} - \|H\|_{L^\infty(F)} \|T \f\|_{L^2(F)}^2 \\
& \geq (1- \e\|H\|_{L^\infty(F)} )  \|\n \f\|^2_{L^2(\O)} + (\L  - \|H\|_{L^\infty(F)} C/\e) \|\f\|^2_{L^2(\O)} \geq \frac{1}{2} \|\f\|^2_{H^1(\O)},
\end{align*}
choosing $\e=\e(C, \|H\|_{L^\infty(F)})>0$ small enough and $\L\geq \L_0(C,\|H\|_{L^\infty(F)})$ large enough. By the Lax-Milgram theorem, the positive definite bilinear form $Q_\L$ then induces a continuous linear operator $\tilde{S}: L^2(\O) \to H^1_F(\O)$, where $\tilde{S}f$ is the unique element of $H^1_F(\O)$, satisfying
\begin{equation}
Q_\L(\tilde{S}f,g) = \langle f, g\rangle_{L^2(\O)} \quad \text{for all } g \in H^1_F(\O),
\end{equation}
for a given $f\in L^2(\O)$. Denoting by $\iota:  H^1_F(\O) \to L^2(\O)$ the compact embedding of $H^1_F(\O)$ into $L^2(\O)$, then $S:=\iota\circ \tilde{S}$ is a compact linear operator on $L^2(\O)$ which is self-adjoint:
\[
\langle f, Sg\rangle_{L^2(\O)} = Q_\L(\tilde{S}f, \tilde{S}g) = Q_\L(\tilde{S}g, \tilde{S}f) = \langle g, Sf\rangle_{L^2(\O)},
\]
and positive since $Q_\L$ is positive definite. By the Spectral Theorem for positive compact self-adjoint operators, we obtain an orthonormal basis for $L^2(\O)$ of eigenfunctions $\{\f_k\}_k$ of $S$ and a corresponding non-increasing sequence $\{\mu_k\}_k\subset (0,\infty)$ of positive eigenvalues: $$S \f_k = \mu_k \f_k,$$
such that $\lim_{k\to\infty} \mu_k = 0$. We now see that $\f_k$ are weak solutions of the eigenvalue problem
\[
\left\{
\begin{aligned}
   -\Delta  \f_k + \L \f_k = \mu_k^{-1} \f_k & \quad \mbox{in } \O, \\
   (\f_k)_\nu - H\phi_k = 0 & \quad \mbox{on }  F, \\
	\f_k= 0 & \quad \mbox{on } \de\O \setminus F,
  \end{aligned}\right.
\]
so that $\l_k:=(\mu_k^{-1} - \L) \nearrow +\infty$ are precisely the eigenvalues of \eqref{eq:eig}. Since $F(u)$ is smooth and $H$ is a smooth function on $F(u)$, the eigenfunctions $\{\f_k\}_k$ are smooth up to $F(u)\cap \de \O$ by local interior and boundary elliptic estimates for the oblique derivative problem (\cite{Liebermanbook}). 

The variational characterization of the eigenvalues in \eqref{thm:std:first} and \eqref{thm:std:Courant-Fischer} now follows by the usual arguments (see \cite[Chapter 22]{Zeidlerbook}). Finally, if $l$ denotes the number of negative eigenvalues of \eqref{eq:eig}, then clearly $\ind(u,\O)\geq l$, as $Q$ is negative definite on $V = \text{Span}(\{\f_k\}_{k=1}^l)$. On the other hand, it must be that $\ind(u,\O)\leq l$, for \eqref{thm:std:Courant-Fischer} implies that in any subspace $V_{l+1}\subset H^1_F(\O)$ with $\dim (V_{l+1}) = l+1$, there is an $L^2$-normalized vector $\psi \in V_{l+1}$, such that $Q(\psi, \psi;\O)\geq \l_{l+1}\geq 0$.
\end{proof}

\begin{prop}\label{prop:firsteig} Under the hypotheses and notation of Theorem \ref{thm:std}, the following properties hold:
\begin{enumerate}[(i)]
\item If $Q(\phi, \phi;\O) = \l_1\|\phi\|^2_{L^2(\O)}$ for some nonzero $\phi\in H^1_F(\O)$, then $\phi$ is an eigenfunction corresponding to $\l_1$.
\item Any eigenfunction, associated to $\l_1$, is strictly positive or strictly negative in $\O$, and the first eigenvalue is simple. 
\item If $D'$ is a bounded subdomain of $D$ and $\O'$ is a connected component of $(D')^+(u)$ such that $\O'\subseteq \O$, then $\l_k(\O')\geq \l_k(\O)$, $k\in \N$. Moreover, if $\O\setminus \overline{\O'}\neq \emptyset$, then the inequality for the first eigenvalue $\l_1$ is strict: $\l_1(\O')>\l_1(\O)$.
\end{enumerate}
\end{prop}
\begin{proof}
\noindent\emph{(i).} Without loss of generality, we may assume that $\|\phi\|_{L^2(\O)}=1$. For any $\psi\in H^1_F(\O)$, we have by \eqref{thm:std:first} that $Q(\phi+t\psi, \phi+t\psi)\geq \l_1\|\phi + t\psi\|_{L^2(\O)}^2$ for $t\in \R$.  Expanding both sides of this equality, we get
\[
Q(\phi,\phi) + 2t Q(\phi, \psi) + t^2 Q(\psi,\psi)\geq \l_1 (1 + 2t \langle \phi,\psi\rangle_{L^2(\O)} + t^2 \|\psi\|^2_{L^2(\O)}).
\]
Subtracting $Q(\phi,\phi) = \l_1$ from both sides, dividing by $t>0$ and taking $t\downarrow 0$, we obtain that $Q(\phi,\psi)\geq  \l_1 \langle \phi,\psi\rangle_{L^2(\O)}$. Doing the same for $t<0$ and taking $t\uparrow 0$, we get the opposite inequality. Thus, we have $Q(\phi,\psi)=  \l_1 \langle \phi,\psi\rangle_{L^2(\O)}$ for any $\psi \in H^1_F(\O)$, meaning that $\phi$ is an eigenfunction associated with $\l_1$. \\

\noindent\emph{(ii).} First note that if an eigenfunction $\f\geq 0$ (resp.\ $f\leq 0$) in $\O$, then by the strong maximum principle it must be that $\f>0$ (resp. $\f<0$) in $\O$. Let $\f_1$ be an eigenfunction associated with $\l_1$, and assume that it changes sign. Then both its positive and negative part $\f_1^{\pm}\in H^1_F(\O)$ are nonzero. Now observe that
\[
Q(\f_1^\pm, \f_1^\pm) = Q(\f_1, \f_1^\pm) = \l_1\langle \f_1,\f_1^+\rangle_{L^2(\O)} = \l_1\|\f_1^\pm\|^2_{L^2(\O)},
\]
so that by \emph{(i)}, $\f_1^\pm$ is an eigenfunction, corresponding to $\l_1$. Hence, both $\f_1^\pm>0$ in all of $\O$, which contradicts the assumption that $\f_1$ changes sign. 

If there are two eigenfunctions $\phi$ and $\psi$ associated to $\l_1$, we can choose them to be $L^2(\O)$-orthogonal. However, that would contradict the strict sign of $\phi$ and $\psi$. \\

\noindent\emph{(iii).} Any function $\psi'\in H^1_F(\O')$ can be extended trivially to a function $\psi\in H^1_F(\O)$, so that $Q(\psi',\psi'; \O') = Q(\psi,\psi,\O)$. By the minimax characterization \eqref{thm:std:first}-\eqref{thm:std:Courant-Fischer}, we then get the monotonicity of the eigenvalue $\l_k(\O')\geq \l_k(\O)$, $k\in \N$. 

Assume that $\O\setminus \overline{\O'} \neq \emptyset$ and let $\psi'\in H^1_F(\O')$ be an eigenfunction of \eqref{eq:eig}, associated with $\l_1(\O')$. Then its extension to $\O$, $\psi \in H^1_F(\O)$, is zero on the nonempty open set $\O\setminus \overline{\O'}$. But if $\l_1(\O')=\l_1(\O)$, then
\[
Q(\psi, \psi, \O) = Q(\psi',\psi';\O') = \l_1(\O')\|\psi'\|^2_{L^2(\O')} = \l_1(\O)\|\psi\|^2_{L^2(\O)},
\]
whence (\emph{i}) implies that $\psi$ is a first eigenfunction in $\O'$. However, (\emph{ii}) then says that $\psi$ is of a strict sign in $\O$, so we reach a contradiction. We conclude that $\l_1(\O')>\l_1(\O)$.
\end{proof}

\section{Global solutions of finite index}\label{sec:globalsol}

In this section, we study global classical solutions $u:\R^n \to [0,\infty)$ of \eqref{FBP} of finite index and connected positive phase $\O:=[\R^n]^+(u)$.  We will show that $u$ must be stable outside a compact set $K$ (Proposition \ref{prop:stab}), so that $\O\setminus K$ admits a positive harmonic function $h$ that satisfies the Robin condition $h_\nu - H h =0$ on $F(u)\setminus K$. We will also prove that $\ind(u)$ is equal to the dimension of a space of globally defined $L^2(\O)$ eigenfunctions of \eqref{eq:eig}. 

We start with some important basic facts about global solutions of the one-phase FBP.

\begin{prop}\label{prop:facts} Let $u$ be a global classical solution of \eqref{FBP}, whose positive phase $\O$ is connected. Then 
\begin{equation}\label{prop:facts:gradbd}
|\n u| \leq 1 \quad \text{in } \O.
\end{equation}
Furthermore, $|\n u|<1$ in $\O$ unless $u$ is a one-plane solution. As a result,
\begin{equation}\label{prop:facts:meancurv}
H(p) = \de_\nu |\n u|^2/2 \geq 0 \quad \text{for all } p\in F(u),
\end{equation}
and $H>0$ on $F(u)$ unless $u$ is a one-plane solution, in which case $H\equiv 0$.
\end{prop}
\begin{proof}
The gradient bound \eqref{prop:facts:gradbd} for global solutions was established in \cite[Proposition A.5]{KW23}. Since $u$ is harmonic in $\O$, we have
\begin{equation}\label{prop:facts:subh}
\D |\n u|^2 = 2 |D^2 u|^2 + 2 \n u \cdot \n (\D u) = 2|D^2 u|^2 \quad \text{in } \O,
\end{equation}
so that $|\n u|^2$ is subharmonic in $\O$. By the strict maximum principle and \eqref{prop:facts:gradbd}, either $|\n u|^2 \equiv 1$ or $|\n u|^2 <1$. In the first case, $\eqref{prop:facts:subh}$ entails that $|D^2 u|^2 \equiv 0$, meaning that $u$ is an affine function in $\O$, whence we conclude that $u$ is a one-plane solution in appropriate Euclidean coordinates, with $H\equiv 0$. In the case of the strict inequality $|\n u|^2 <1$, we have on account of the Hopf Lemma:
\[
H = \text{div}\frac{\n u}{|\n u|}=  \frac{1}{2} \frac{(-\n u) \cdot \n |\n u|^2}{|\n u|^3} = \frac{1}{2}\de_\nu |\n u|^2 > 0 \quad \text{on } F(u),
\]
since $-\n u = \nu$ and $|\n u|=1$ on $F(u)$.
\end{proof}

For the constructions in this section it will be useful to know that the positive phase minus a large enough ball, $\O\setminus \overline{B_\rho}$, has finitely many connected components. In order to prove this, we need an auxiliary technical result that says that for any $\rho_0>0$, we can always find $\rho>\rho_0$, such that either $\de B_\rho$ doesn't intersect $F(u)$, or $\de B_\rho$ intersects $F(u)$ transversally.

\begin{lemma}\label{lem:Sard}
Let $u$ be a classical solution to the one-phase FBP in $\R^n$ with a connected positive phase $\O$ and let $\rho_0>0$. Then there exists $\rho_1>\rho_0$ such that either $\de B_{\rho_1}\cap F(u) = \emptyset$, or $\de B_{\rho_1}$ intersects $F(u)$ transversally. In the second case, one can find $\rho_2>\rho_1$ such that $\de B_{\rho_2}$ intersects $F(u)$ transversally, as well. 
\end{lemma}
\begin{proof}
Consider the smooth function $f: F(u)\to [0,\infty)$ on the smooth submanifold $F(u)$ given by $f(p)=|p|^2$. We observe that the intersection of $\de B_\rho$ and $F(u)$ is empty or transversal if and only if $\rho^2$ is not a critical value of $f$. 

Let $\g$ be a connected component of the free boundary $F(u)$. Then either $\g$ is compact, or $\g$ is an unbounded closed set. Thus, there are three possibilities for $f(\g)$:
\begin{enumerate}
\item $f(\g)= \{a^2\}$ for some $a\geq 0$;
\item $f(\g)= [a^2,b^2]$ for some $0\leq a<b<\infty$;
\item $f(\g) = [a^2, \infty)$ for some $a\geq 0$.
\end{enumerate}

We claim that in the first case, it must be that $a>0$ and $u(x) = [G(x)]^+$, where 
\[
G(x) = \begin{cases} a \log|x/a| & \text{when } n=2, \\  a (1-|x/a|^{2-n}) & \text{when } n=3. \end{cases}
\]
Indeed, and for any $p\in f^{-1}(a^2)$, a neighbourhood $F(u)\cap B_r(p)$ of $p$ in $F(u)$ must be a piece of the sphere $\de B_a$, $a>0$. But since $u=0$ and $|\n u|=1$ on $F(u)\cap B_r(p) = \de B_a \cap B_r(p)$, the Cauchy-Kovalevskaya theorem and the unique continuation principle, applied to the harmonic $u:\O\to [0,\infty)$, yield that 
\[
u(x) = |G(x)| \quad \text{for all } x\in \O,
\]
As $\O$ is unbounded and connected, it must be that $\O= (B_a)^c$ and $u = G^+$. In particular, this means that $\de B_{\rho_1}\cap F(u)=\emptyset$ for all $\rho_1>a$.

Hence, we may now assume that for any connected component $\g$ of $F(u)$, $f(\g)$ is a nontrivial closed interval or a closed semi-infinite interval.  If $F(u)\setminus \overline{B_{\rho_0}}=\emptyset$, then $\de B_{\rho_1}$ doesn't intersect $F(u)$ for any $\rho_1>\rho_0$ and we are done. If not, there exists a component $\g$ that has a nontrivial intersection with $(\overline{B_{\rho_0}})^c$. In particular, $f(F(u))$ is guaranteed to contain a nontrivial open interval $I\subset (\rho_0^2, \infty)$. By Sard's theorem a.e. $\rho^2 \in I$ is not a critical value of $f$. Therefore, there exists $\rho_2>\rho_1 > \rho_0$ such that both $\de B_{\rho_1}$ and $\de B_{\rho_2}$ intersect $F(u)$ transversally. 
\end{proof}

\begin{prop}\label{prop:finend} Let $u$ be a global classical solution of \eqref{FBP} in $\R^n$, whose positive phase $\O$ is connected. Assume that $\de B_\rho$ intersects $F(u)$ transversally. Then $\O\setminus \overline{B_\rho}$ has finitely many connected components. 
\end{prop}
\begin{proof}
By Lemma \ref{lem:Sard} there exists $\rho_1>\rho$ such that $\de B_{\rho_1}$ again intersects the smooth hypersurface $F(u)$ transversally. As $\O$ is connected, and $0\in \O$, every connected component of $\O\setminus \overline{B_\rho}$ has a non-empty intersection with $\O\cap (B_{\rho_1}\setminus \overline{B_\rho}),$ so we will be done once we show that $\mathcal{E}:=\{E: E \text{ is a connected component of }\O\cap (B_{\rho_1}\setminus \overline{B_\rho})\}$ is finite.

First, claim that a point $p\in F(u)$ can belong to the boundary of at most one member of $\mathcal{E}$. Fix $p \in F(u)\cap \overline{B_{\rho_1}\setminus B_\rho}$, and let $U_r(p):=[B_r(p)]^+(u)$. For $r>0$ small enough, $U_r(p)$ is the supergraph of a smooth function in the direction of $\n u(p)$ in $B_r(p)$ over the tangent hyperplane. It suffices to show that $U_r(p)\cap (B_{\rho_1}\setminus \overline{B_\rho})$ is non-empty and connected for some small enough $r>0$. There are two possibilities for the position of $p$. First, if $p\in F(u)\setminus (\de B_\rho\cup \de B_{\rho_1})$, then $U_r(p)\cap (B_{\rho_1}\setminus \overline{B_\rho}) = U_r(p)$ is certainly connected for all small enough $r$. The other possibility is that $p\in \de F(u) \cap \de B_\rho$ or $p\in \de F(u) \cap \de B_{\rho_1}$. If $p\in \de F(u) \cap \de B_\rho$, then the transversal intersection of $\de B_\rho$ with $F(u)$ means that $\de B_\rho$ separates $U_r(p)$, for small enough $r>0$, into two connected components $U_r^\pm$: 
\[
U_r(p)\setminus \de B_\rho = U_r^+ \cup U_r^-, \quad \text{where} \quad U_r^+\subset (\overline{B_\rho})^c \text{ and }   U_r^-\subset B_\rho. 
\]
Thus, $U_r(p)\cap (B_{\rho_1}\setminus \overline{B_\rho}) = U_r^+$ is again connected. The case, where $p\in \de F(u) \cap \de B_{\rho_1}$, is treated analogously. 

Now if $\mathcal{E}$ contains infinitely many distinct members $\{E_k\}_{k=1}^\infty$, we can choose  points $p_k\in \de E_k\cap F(u)$, $k\in \N$. Since the $\{p_k\}$ belong to the compact $F(u)\cap \overline{B_{\rho_1}\setminus B_\rho}$, up to taking a subsequence, they converge to a point $p_\infty \in F(u)\cap \overline{B_{\rho_1}\setminus B_\rho}$. Using the argument  above, we see that $U_r(p_\infty)\cap (B_{\rho_1}\setminus \overline{B_\rho})$ is always non-empty and connected for a small enough $r>0$, so that $p_\infty\in \de E_\infty$ for some $E_\infty\in \mathcal{E}$. Furthermore, for this value of $r>0$, the free boundary subset $F_r:=[F(u)\cap \overline{B_{\rho_1}\setminus B_\rho}]\cap B_r(p_\infty)$ belongs to $\de E_\infty$. However, for all large $k$, we obviously have $p_k\in F_r$, so that by our claim above, $p_k$ can only belong to $\de E_\infty$ and to the boundary of no other member of $\mathcal{E}$. This contradicts the fact that the $p_k$ were chosen from the boundaries of distinct members $E_k$ of $\mathcal{E}$.
\end{proof}

We are now ready to establish the key result that a finite index classical solution of $u$ in $\R^n$ must be stable outside a compact set. 

\begin{prop}\label{prop:stab} Let $u:\R^n\to [0,\infty)$ be a global solution of \eqref{FBP} with a connected positive phase  $\O=\{u>0\}$ and let $0\in \O$. If $u$ has finite Morse index, then $u$ is stable in $\R^n\setminus \overline{B_\rho}$, for some large enough $\rho>0$. Moreover, there exists a positive function $h\in C^\infty(\overline{\O}\setminus \overline{B_\rho})$ solving
\begin{equation}\label{prop:stab:eq}
\left\{
\begin{aligned}
   -\Delta  h = 0 & \quad \mbox{in } \O\setminus \overline{B_\rho}, \\
   h_\nu - H h = 0 & \quad \mbox{on }  F(u)\setminus  \overline{B_\rho}. \\
  \end{aligned}\right.
\end{equation}
\end{prop}

\begin{proof}
Denote by $\O_r$ the connected component of $B_r^+(u)$ containing the origin and note that $\O_{r_1}\subseteq \O_{r_2}$ whenever $r_1<r_2$. Let $k:=\ind(u)<\infty$. By \eqref{eq:indlim} there exists $\rho$ sufficiently large that $k=\ind(u, \O_\rho)$. According to Lemma \ref{lem:Sard}, we can choose $\rho$ in a way that the intersection of $\de B_\rho$ with $F(u)$ is either transversal or empty.

We claim that $u$ is stable in $D:=\R^n \setminus \overline{B_\rho}$. If not, there exists a nontrivial function $\phi\in \X_{D^+(u)}$ such that $Q(\phi, \phi; D^+(u)) < 0.$ 
By the connectedness of $\O$, we can find a second radius $R>\rho$ such that $\text{supp }\phi \cap \O \subseteq \O_R,$
meaning that $$Q(\phi, \phi;\O_R)=Q(\phi, \phi; D^+(u)) <0.$$
Now, let $\{\f_i\}_{i=1}^k \subseteq H^1_F(\O_{\rho})$ be $k$  be orthonormal eigenfunctions of \eqref{eq:eig} in $\O_\rho$, associated with all the negative eigenvalues $\{\l_i\}_{i=1}^k$.  Extend them trivially to $H^1_F(\O_R)$-functions. Then
\[
Q(\f_i,\f_j; \O_R) = \l_i \d_{ij} \quad \text{for all } i,j =1,2,\ldots k,
\]
and since $\f_i$  and $\phi$ have disjoint supports in $\O_R$, we also have 
$$
Q(\f_i, \phi;\O_R)=0 \quad \text{for all } i=1,\ldots, k.
$$
Hence, if $W=\text{Span}(\{\phi\} \cup \{\f_i\}_{i=1}^k)$, then $\dim(W)=k+1$ and we see that for any nontrivial $\psi= c_0\phi + \sum_{i=1}^{k} c_i \f_i \in W\setminus \{0\}$,
\[
Q(\psi,\psi;\O_R) = c_0^2 Q(\phi, \phi;\O_R) + \sum_{i=1}^k c_i^2 Q(\f_i,\f_i;\O_R) < 0
\]
meaning that $\ind(u,\O_R)\geq k+1$, which contradicts the fact that $\ind(u,\O_R)\leq \ind(u)= k$.

Let us now construct a positive function $h\in C^\infty(\overline{\O}\setminus \overline{B_\rho})$ satisfying \eqref{prop:stab:eq}. If $F(u)\cap \de B_\rho =\emptyset$, then $\O\setminus \overline{B_\rho}$ is connected. If not, $\de B_\rho$ intersects $F(u)$ transversally, so that $\O\setminus \overline{B_\rho}$ again consists of finitely many connected components  $\{\O_j\}_{j=1}^m$, according to Proposition \ref{prop:finend}. We will be done once we construct $h$ in each $\O_j.$ 

Fix $j\in \{1,\ldots, m\}$ and denote $U:=\O_j$.  Let $R>\rho$ and denote by $U_R$ the connected component of $U\cap B_R$ that borders $\de B_\rho$. Since $u$ is stable in $\R^n\setminus \overline{B_\rho}$, then $Q(\cdot, \cdot; U_R)$ is nonnegative definite, so that the first eigenvalue $\l_1(U_R) \geq 0$ for all $R>\rho$  by \eqref{thm:std:first} of Theorem \ref{thm:std}. Furthermore, as $U_R \setminus \overline{U_{R'}}\neq \emptyset$ for $\rho<R'<R$, Proposition \ref{prop:firsteig}(\textit{iii}) yields that 
\begin{equation*}
\l_1(U_R)>\l_1(U_{R'})\geq 0 \quad \text{for all } R>\rho.
\end{equation*}
Because of the strict positivity of the first eigenvalue $\l_1(U_{R})>0$, the Fredholm alternative tells us that the problem 
\[
\left\{
\begin{aligned}
   -\Delta  v = |D^2 u|^2 & \quad \mbox{in } U_{R}, \\
   v_\nu - H v = 0 & \quad \mbox{on }  F(u)\cap \de U_{R}, \\	
  v = 0 & \quad \mbox{on } \de U_{R} \setminus F(u),
  \end{aligned}\right.
\]
with a right-hand side  $|D^2 u|^2\in L^2(U_R)$, 
has a unique weak solution $v_{R}\in H^1_F(U_{R})$. Noting that $w:= (|\n u|^2 + 1)/2\geq 1/2$ satisfies
\begin{align*}
\D w = |D^2 u|^2  \text{ in } \O \quad \text{and} \quad w_\nu  = H = H w \text{ on } \de \O,
\end{align*}
on account of \eqref{prop:facts:meancurv}, we get that $h_{R}:= v_{R} + w \in H^1(U_{R})$ is a weak solution to
\begin{equation*}
\left\{
\begin{aligned}
   -\Delta  h = 0 & \quad \mbox{in } U_{R}, \\
   h_\nu - H v = 0 & \quad \mbox{on }  F(u)\cap \de U_{R}, \\	
  h = w & \quad \mbox{on } \de U_{R} \setminus F(u)
  \end{aligned}\right.
\end{equation*}
in the sense that $h_{R}-w\in H^1_F(U_{R})$ and 
\begin{equation}\label{eq:weakIR}
Q(h_{R}, \psi; U_{R}) = 0 \quad \text{for any }\psi \in H^1_F(U_{R}).
\end{equation}
Claim that the harmonic function $h_{R}>0$ in $U_{R}$. By the strong maximum principle, it suffices to show that $h_{R}\geq 0$ in $U_{R}$. If $h_{R}$ changed sign, the fact that $h_{R}=w>0$ on $\de U_{R} \setminus F(u)$ would imply that its nonzero negative part $h^-_{R}\in H^1_F(U_{R})$. But then \eqref{eq:weakIR} entails that
$$
Q(h^-_{R},h^-_{R}; U_{R}) = Q(h_{R},h^-_{R}; U_{R}) = 0,
$$
which would contradict the strict stability estimate
\[
Q(h^-_{R},h^-_{R}; U_{R})  \geq \l_1(U_{R}) \|h^-_{R}\|_{L^2(U_{R})}^2 > 0. 
\]

Now, fix a point $p$ that is contained in $U_{R}$ for all $R>\rho$ and define $\bar{h}_{R}(x):=\frac{h_{R}(x)}{h_{R}(p)}$ in $U_{R}$. Note that $\bar{h}_{R}\in H^1(U_{R})$ solves
\[
\begin{aligned}
   -\Delta  h = 0 & \quad \mbox{in } U_{R}, \\
   h_\nu - H v = 0 & \quad \mbox{on }  F(u)\cap \de U_{R}, \\	
  h > 0 & \quad \text{in } U_{R} \quad \text{with } h(p)=1.
  \end{aligned}
\]
Let $A$ and $A_1$, with $A\Subset A_1$, be two annuli, containing $p$, that are centered at the origin and compactly contained in $B_R\setminus \overline{B_\rho}$ for all large $R>0$. By the Harnack inequality (see \cite[Theorem 5.44]{Liebermanbook}), we know that $\sup_{A_1\cap U_{R}}\bar{h}_{R}(x) \leq C_{A_1}$ is bounded independently of $R$, so that by local boundary elliptic estimates
\[
\|\bar{h}_{R}\|_{C^l(A\cap U_{R})} \leq C_{l,A,A'} \quad \text{for all large } R,
\]
since $F(u)$ is smooth and  $H\in C^\infty(F(u))$. Hence, we can find a sequence $R_i\uparrow \infty$ and a harmonic function $h\in C^\infty(\overline{U}\setminus \de B_\rho)$ such that for any $l\in\N$,
$$
\bar{h}_{R_i} \to h \quad \text{uniformly in } C^{l}(A \cap U),
$$
for any annulus $A\Subset \R^n \cap \overline{B_\rho}$, centered at the origin. We conclude that $h$  satisfies $(h_i)_\nu - H h_i = 0$ classically on $\de U \setminus \de B_\rho$  and that $h(p)=1$. The latter implies that $h$ is positive in $U$ due to the connectedness of $U$ and the strong maximum principle.

\end{proof}

In the next proposition we characterize the finite index of a global solution $u$ with a positive phase $\O$ as the maximal number of linearly independent $L^2(\O)$-eigenfunctions associated to negative eigenvalues.

\begin{prop}\label{prop:L2} Let $u:\R^n\to [0,\infty)$ be a global solution of \eqref{FBP} with a connected positive phase $\O=\{u>0\}$. The following are equivalent:
\begin{enumerate}[(i)]
\item $\ind(u)<\infty$;
\item There exists a finite dimensional subspace $V\subseteq L^2(\O)$, generated by an orthonormal set $\{\f_i\}_{i=1}^k$ of eigenfunctions of the problem
\[
\begin{aligned}
   -\Delta  \f_i = \l_i \f_i & \quad \mbox{in } \O, \\
   (\f_i)_\nu - H \f_i = 0 & \quad \mbox{on }  F(u),
  \end{aligned}
\]
associated to negative eigenvalues $\{\l_i\}_{i=1}^k$, such that $Q(\phi, \phi;\O)\geq 0$ for all $\phi \in C^\infty_c(\overline{\O})$ that satisfy $\phi\lfloor_\O \in V^\perp$.
\end{enumerate}
Moreover, if $\ind(u)<\infty$, then $\ind(u)=\mathrm{dim}(V)$.
\end{prop}
\begin{proof}
(ii) $\Rightarrow$ (i). Assume to the contrary that $\ind(u)=\infty$. This means that for some large enough $R>0$, we would have $\ind(u;\O_R)>k=\text{dim}(V)$. Then there exists an $(k+1)$-dimensional subspace $\tilde{W}\subseteq  H^1_F(\O_R)$ such that $Q(\phi, \phi; \O_R)<0$ for all $\phi\in \tilde{W}\setminus\{0\}$. Treating the quadratic form $[-Q](\cdot, \cdot;\O_R)$ as an inner product on $\tilde{W}$, we can find a $[-Q]$-orthonormal basis $\{ \tilde{\phi}_i\}_{i=1}^{k+1}$ for $\tilde{W}$:
\[
-Q(\tilde{\phi}_i, \tilde{\phi}_j; \O_R) = \d_{ij} \quad \text{for } i,j\in \{1,\ldots, k+1\}.
\]
By approximating each $\tilde{\phi}_i\in H^1_F(\O_R)$ with a function $\phi_i\in \X_{\O_R}$, we can find, for any $\e>0$, a set $S:=\{\phi_i\}_{i=1}^{k+1}\subset  \X_{\O_R}$ of $(k+1)$ functions that are almost $[-Q(\cdot,\cdot;\O_R)]$ orthonormal:
\begin{equation}\label{prop:L2:almostorto}
|-Q(\phi_i, \phi_j; \O_R) - \d_{ij}| \leq \e \quad \text{for } i,j\in \{1,\ldots, k+1\}.
\end{equation}
Hence, $S$ is a linearly independent set and $Q(\cdot, \cdot, \O_R)$ is negative definite on $W:=\text{Span}(S)\subset \X_{\O_R}$, for small enough $\e>0$.

Now, as $\text{dim}(W)=k+1>k$, there exists a nontrivial element $\phi\in W \subset C^\infty_c(\overline{\O})$, such that
\[
\int_\O \phi \f_i \, dx = 0 \quad \text{for } i = 1, \ldots, k. 
\]
Nevertheless, $Q(\phi,\phi;\O)=Q(\phi,\phi;\O_R)<0$, which yields a contradiction. 

\bigskip
\noindent (i) $\Rightarrow$ (ii). Let $\rho_0>0$ be sufficiently large such that 
\begin{itemize}
\item $\ind(u,\O_R)=\ind(u)=:k$ for all $R\geq \rho_0$, and 
\item $u$ is stable in $\R^n\setminus \overline{B_\rho}$, according to Proposition \ref{thm:std}. 
\end{itemize}
Let $R\geq \rho_0$ and denote by $\{\l_{i,R}:=\l_i(\O_R)\}_{i=1}^k$ the negative eigenvalues of \eqref{eq:eig} in $\O_R$ in increasing order, counted with multiplicity, and by $\{\f_{i,R}\}_{i=1}^k$ their corresponding $L^2(\O_R)$-normalized eigenfunctions. By Theorem \ref{thm:std}, we know that $\{\f_{i,R}\}_{i=1}^k \subset C^\infty(\overline{\O_R}\setminus \de B_R)$ and form an orthonormal set in $L^2(\O_R)$. We also have from Proposition \ref{prop:firsteig}(\emph{iii}) that $\l_{k,R}$ is a decreasing function of $R$, which implies that
\begin{equation}\label{prop:L2:lambda}
\l_{i,R} \leq \l_{k,R} \leq \l_{k,\rho_0}=: -c_0 < 0 \quad \text{for } i=1,\ldots, k, 
\end{equation}
where $c_0$ depends on $u$ only. 

\emph{Claim that for all $\rho_0\leq \rho < 2\rho \leq R$, we have} 
\begin{equation}\label{prop:L2:est1}
\int_{\O_R\setminus B_{2\rho}} \f_{i,R}^2 \, dx \leq C \rho^{-2},
\end{equation}
for some constant $C$ depending on $u$ only.
\begin{proof}
Take a standard smooth cut-off function $\bar{\eta}: \R^n \to [0,1]$ such that $\bar{\eta} = 0$ in $B_\rho$, $\bar{\eta} = 1$ in $\R^n\setminus B_{2\rho}$ and $\|\n \bar{\eta}\|_{L^\infty(\R^n)}\leq 2/\rho$. Since $u$ is stable in $\R^n\setminus \overline{B_\rho}$, we have that for $\f:=\f_{i,R}$, 
\begin{align}\label{prop:L2:stabty}
0\leq Q(\bar{\eta} \f, \bar{\eta} \f; \O_R) &= \int_{\O_R} |\n (\bar{\eta} \f)|^2 \, dx - \int_{\de \O_R} H (\bar{\eta} \f)^2 \, d\H^{n-1} \notag  \\
& =\int_{\O_R} |\n \bar{\eta}|^2 \f^2 \, dx + \int_{\O_R} (\n \f  \cdot (\f \n \bar{\eta}^2)+  |\n \f|^2\bar{\eta}^2) \, dx - \int_{\de \O_R} H (\bar{\eta} \f)^2 \, d\H^{n-1} \notag \\
& = \int_{\O_R} |\n \bar{\eta}|^2 \f^2 \, dx + Q(\f, \bar{\eta}^2\f ; \O_R). 
\end{align}
On the other hand, since the eigefunction $\f$ satisfies $Q(\f, \phi) = \l_{i,R} \langle \f, \phi\rangle_{L^2(\O_R)}$ for any test function $\phi \in H^1_F(\O_R)$, we have $Q(\f, \bar{\eta}^2\f ) = \l_{i,R} \|\bar{\eta}\f\|^2_{L^2(\O_R)}$. Thus, \eqref{prop:L2:stabty} implies
\begin{equation*}
-\l_{i,R} \|\bar{\eta}\f\|^2_{L^2(\O_R)} = -Q(\f, \f \bar{\eta}^2; \O_R) \leq \int_{\O_R} |\n \bar{\eta}|^2 \f^2 \, dx \leq 4 \rho^{-2}\|\f\|^2_{L^2(\O_R)} =  4 \rho^{-2},
\end{equation*}
so that \eqref{prop:L2:lambda} yields
\begin{equation}
\|\bar{\eta}\f\|^2_{L^2(\O_R)} \leq 4 \rho^{-2}/(-\l_{i,R}) \leq (4/c_0) \rho^{-2} = C\rho^{-2}. 
\end{equation}
Now, we can conclude the desired estimate \eqref{prop:L2:est1}:
\[
\int_{\O_R\setminus B_{2\rho}} \f^2 \, dx \leq \int_{\O_R} (\bar{\eta}\f )^2 \, dx \leq C\rho^{-2}. 
\]
\end{proof}

\emph{Claim that for all $\rho_0\leq \rho < 2\rho \leq R$, we have} 
\begin{equation}\label{prop:L2:est2}
\int_{B_{\rho}\cap \O_R} (\f_{i,R}^2 + |\n \f_{i,R}|^2) \, dx \leq c
\end{equation}
for some constant $c=c(\rho, u)$ depending on $\rho$ and $u$ only.
\begin{proof}
Again, denote $\f:=\f_{i,R}$ and let $\bar{\eta}$ be the cut-off function from above. Define $$\eta:=1-(1-\bar{\eta})^2.$$ 
Then $\eta$ is a smooth cut-off function such that $0\leq \eta\leq 1$, $\eta = 0$ in $B_\rho$, $\eta = 1$ in $\R^n\setminus B_{2\rho}$, and 
\begin{equation}\label{prop:L2:cutoff}
|\n \eta|^2= 4(1-\bar{\eta})^2|\n \bar{\eta}|^2 \leq 16 \rho^{-2} (1-\eta) \leq 16\rho^{-2} (1-\eta^2). 
\end{equation}
In order to establish \eqref{prop:L2:est2}, it suffices to show that 
\begin{equation}\label{prop:L2:est2red}
\int_{B_\rho\cap \O_R}|\n \f|^2 \, dx \leq \bar{c}(\rho,u), 
\end{equation}
as $\int_{B_\rho\cap \O_R} \phi^2\, dx \leq \|\phi\|^2_{L^2(\O_R)}=1$.
Since $u$ is stable in $\R^n\setminus \overline{B_\rho}$, we have $Q(\eta \f, \eta\f)\geq 0$, so that
\begin{equation}\label{prop:L2:difQ1}
Q(\f, \f) - Q(\eta \f, \eta\f) \leq Q(\f, \f) = \l_{i,R} < 0. 
\end{equation}
On the other hand, since $2 \eta\f \n\eta \cdot \n \f \leq \eps \rho^2 |\n\eta|^2 |\n \f|^2 + \frac{1}{\eps \rho^2} \eta^2\f^2$ for all $\eps>0$, we have
\begin{align}\label{prop:L2:difQ2}
\l_{i,R} &\geq Q(\f, \f) - Q(\eta \f, \eta\f) = \notag \\
&=  \int_{\O_R} (1-\eta^2) |\n \f|^2  \, dx - \int_{\de \O_R} (1-\eta^2) H \f^2 \, d\H^{n-1} 
- \int_{\O_R} \left( |\n \eta|^2 \f^2 + 2 \eta\f \n\eta \cdot \n \f \right)\, dx \notag  \\
&\geq \int_{\O_R} (1-\eta^2 - \eps \rho^2 |\n\eta|^2) |\n \f|^2 \, dx -  \int_{\O_R} \left( \frac{\eta^2}{\eps \rho^2} + |\n \eta|^2\right) \f^2   \, dx \\
& - \int_{\de \O_R} (1-\eta^2) H \f^2  d\H^{n-1}. \notag
\end{align}
Taking $\eps = 1/32$ and using \eqref{prop:L2:cutoff}, we can bound from below the right-hand side of \eqref{prop:L2:difQ2} by
\begin{equation}\label{prop:L2:penult}
\l_{i,R} \geq \int_{\O_R} \frac{1-\eta^2}{2} |\n \f|^2 \, dx  - C_1 \rho^{-2} \|\f\|^2_{L^2(\O_R)} - \max_{F(u)\cap B_{2\rho}} |H| \int_{\de \O_R} (1-\eta^2) \f^2 \, d \H^{n-1}. 
\end{equation}
for some numerical constant $C_1$. Furthermore, applying the inequality \eqref{prop:trace:L1} to $\psi = (1-\eta^2)\f$ in $\O_R$ (and noting that the Lipschitz constant of $u$ is $L=1$), we get
\begin{align}\label{prop:L2:Sobest}
\int_{\de \O_R} (1-\eta^2)\f^2 \, d \H^{n-1} & \leq  \int_{\O_R} |\n \left((1-\eta^2)\f^2\right)| \, dx \notag \\
& \leq \int_{\O_R} 2 \eta |\n \eta| \f^2 \, dx + \int_{\O_R} (1-\eta^2) 2 |\f| |\n \f| \, dx \notag \\
& \leq C_2\rho^{-1} \|\f\|^2_{L^2(\O_R)} + \e^{-1} \|(1-\eta^2)^{1/2} \f\|^2_{L^2(\O_R)} + \e \int_{\O_R} (1-\eta^2) |\n \f|^2 \, dx \notag \\
& \leq  \left(\frac{C_2}{\rho} + \frac{1}{\e} \right)\|\f\|^2_{L^2(\O_R)} + \e \int_{\O_R} (1-\eta^2) |\n \f|^2 \, dx,
\end{align}
for some numerical constant $C_2$.
Plugging \eqref{prop:L2:Sobest} in \eqref{prop:L2:penult}, and taking into consideration that $\|\f\|_{L^2(\O_R)}=1$, we get 
\begin{align}\label{prop:L2:almost}
\l_{i,R} \geq \int_{\O_R} \left(\frac{1}{2} -\e \max_{F(u)\cap B_{2\rho}} |H| \right) (1-\eta^2) |\n \f|^2 \, dx  - \frac{C_1}{\rho^2} - \left(\frac{C_2}{\rho} + \frac{1}{\e}\right) \max_{F(u)\cap B_{2\rho}} |H|. 
\end{align}
Fixing $\e>0$ small enough so that $\e \max_{F(u)\cap B_{2\rho}} |H| \leq 1/4$, \eqref{prop:L2:almost} yields 
\begin{equation}\label{prop:L2:lambdaboundbelow}
\l_{i,R} \geq \int_{\O_R} \frac{1-\eta^2}{4} |\n \f|^2 \, dx  - c_1,
\end{equation}
for some constant $c_1$ that depends only on $u$ and $\rho$. In particular, if $c_2:=c_1(u,\rho_0)$, then
\begin{equation}\label{prop:L2:boundseig}
-c_2 \leq \l_{i,R} \leq -c_0, 
\end{equation}
and the bounds $c_0$ and $c_2$ depend on $u$ only. Furthermore, as $\l_{i,R}<0$, \eqref{prop:L2:lambdaboundbelow} yields the desired bound \eqref{prop:L2:est2red}:
\[
\int_{\O_R \cap B_\rho} |\n \f|^2\, dx \leq \int_{\O_R} (1-\eta^2) |\n \f|^2 \, dx \leq 4 (c_1 + \l_{i,R}) \leq 4c_1=:\bar{c}. 
\]
\end{proof}
Because of \eqref{prop:L2:est1}, \eqref{prop:L2:est2} and the compactness of both the inclusion $H^1_F(\O_R) \subset L^2(\O_R)$ and the trace operator $H^1_F(\O_R)\to L^2(\de \O_R \cap F(u))$, we can find a sequence $R_l\uparrow \infty$ and globally defined limiting functions $\{\f_i\}_{i=1}^k\subseteq L^2(\O)\cap H^1_{F,\text{loc}}(\O)$ such that as $l\to\infty,$
\begin{align}
\f_{i,R_l} &\to \f_i \quad \text{in } L^2(\O), \quad \label{prop:L2:conv1}\\
\f_{i,R_l} & \to \f_i  \quad \text{in } L^2_\text{loc}(\de \O) \quad \text{and} \quad \n \f_{i,R_l} \rightharpoonup \n \f_i \label{prop:L2:conv2} \quad \text{weakly in } L^2_\text{loc}(\O). 
\end{align}
As a result of \eqref{prop:L2:conv1}, we see that $\{\f_i\}_{i=1}^k$ is an $L^2$-orthonormal set, as well. Furthermore, as each eigenvalue $\l_{i,R}$ is a decreasing function of $R$ that is also bounded \eqref{prop:L2:boundseig}, we have the convergence of the eigenvalues, as well:
\begin{equation}\label{prop:L2:eigconv}
\l_{i,R_l} \searrow \l_i \quad \text{as } l\to\infty, \quad \text{for some }  \l_i\in [-c_2, -c_0].
\end{equation}
Therefore, by \eqref{prop:L2:conv1}, \eqref{prop:L2:conv2} and \eqref{prop:L2:eigconv}, we see that for every $\phi\in C^\infty_c(\overline{\O})$
\[
Q(\f_i, \phi) - \l_i \langle \f_i,\phi\rangle_{L^2(\O)} = \lim_{l\to\infty}\left(Q(\f_{i,R}, \phi; \O_{R_l}) - \l_{i,R_l} \langle \f_{i,R_l},\phi\rangle_{L^2(\O_{R_l})} \right) = 0. 
\]
It follows that $\{\f_i\}_{i=1}^k$ is an $L^2(\O)$-orthonormal set of eigenfunctions in $\O$, with associated eigenvalues $\{\l_i\}_{i=1}^k$.  Denote their linear span by $V$. Then $\dim ~V = \ind(u) = k$ and take any $\phi \in C^\infty_c(\overline{\O})$ such that $\phi \in V^\perp$. For all $R$ large enough such that $B_R\supseteq \text{supp }\phi$, define
\[
\phi_R := \phi - \sum_{i=1}^k a_{i,R} \f_{i,R} \quad \text{in } H^1_F(\O_R), \text{ where } a_{i,R}:=\langle \phi, \f_{i,R} \rangle_{L^2(\O_R)}.
\]
Thus, $\phi_R$ is orthogonal to the span of $\{\f_{i,R}\}_{i=1}^k$, so that $Q(\phi_R,\phi_R; \O_R)\geq 0$ and 
$$
Q(\phi_R, \f_{i,R};\O_R) = \l_{i,R} \langle \phi_R, \f_{i,R}\rangle_{L^2(\O_R)} = 0. 
$$
As a consequence,
\begin{align}\label{prop:L2:final}
Q(\phi, \phi)  & = Q(\phi_{R_l}, \phi_{R_l}; \O_{R_l}) +  \sum_{i,j} a_{i,R_l} a_{j,R_l} Q(\f_{i, R_l}, \f_{j,R_l}; \O_{R_l}) \geq \sum_{i=1}^k  a_{i,R_l}^2 \l_{i,R_l}.
\end{align}
Now we just observe that 
$$\lim_{l\to\infty} a_{i,R_l} = \langle \phi, \f_{i}\rangle_{L^2(\O)} = 0$$ on account of the $L^2$-convergence \eqref{prop:L2:conv1}. Since by \eqref{prop:L2:boundseig}, the eigenvalues $\l_{i,R}$, $i=1,\ldots, k$, are uniformly bounded for large $R$, \eqref{prop:L2:final} implies that $Q(\phi,\phi)\geq 0$ by letting $l\to\infty$.
\end{proof}

\bigskip

\section{Finite index is finite topology in dimension two}\label{sec:dimtwo}

In this section, we zero in on the global, finite-index solutions of the one-phase FBP in $\R^2$ and prove that they have finite topology. We first show that such solutions satisfy a \emph{finite total mean curvature bound.} 

\begin{lemma}\label{lem:meancurvest}
Let $u:\R^2 \to [0,\infty)$ be a global solution of \eqref{FBP} with a connected positive phase $\O=\{u>0\}$ and a finite Morse index. Then there exists $\rho>0$ such that
\begin{equation}\label{lem:curvbd:eq}
\int_{F(u)\setminus \overline{B_{2\rho}}} H \, d \H^1 \leq3 \pi.
\end{equation}
In particular, $\int_{F(u)} H \, d \H^1 < \infty$ (note that by Proposition \ref{prop:facts}, we have $H\geq 0$). 
\end{lemma}
\begin{proof}
Since $\ind(u)<\infty$, Proposition \ref{thm:std} implies that $u$ is stable in $\R^n\setminus \overline{B_{\rho/2}}$  for some $\rho>0$.  Define the following logarithmic cutoff test function $\phi_R\in C^{0,1}_c(\R^2)$ whose $\text{supp}(\phi_R)\Subset  \R^n\setminus \overline{B_{\rho/2}}$:
\[
\phi_R(x) = \begin{cases}
0 & \text{for } |x| \leq \rho, \\ 
(|x| - \rho)/\rho & \text{for } \rho< |x| \leq 2\rho, \\
1 &  \text{for } 2\rho< |x| \leq R, \\
2 - (\log |x|)/\log R & \text{for } R< |x| \leq R^2, \\
0 & \text{when } |x|>R^2. 
\end{cases}
\]
We deduce by the stability of $u$ in $\R^n\setminus \overline{B_{\rho/2}}$ that for all $R$ large enough, 
\begin{align*}
\int_{F(u)\cap (B_{R}\setminus \overline{B_{2\rho}})} H \, d\H^1 & \leq \int_{F(u)} H \phi_R^2 \, d\H^1 \leq \int_{\O} |\n \phi_R|^2 \, dx \leq  \int_{\R^2} |\n \phi_R|^2 \, dx  \\ 
& \leq \int_{B_{2\rho}\setminus B_\rho} \frac{1}{\rho^2} \, dx + \int_{B_{R^2} \setminus B_R} \frac{1}{(\log^2 R)|x|^2}\,dx  = 3 \pi + \frac{2\pi}{\log R}.
\end{align*}
Taking $R \to \infty$ above, we get precisely \eqref{lem:curvbd:eq}.

Now for every $p\in F(u)\cap \overline{B_{2\rho}}$, there exists a small enough disk $B_r(p)$ such that $F(u)\cap \overline{B_r(p)}$ is a graph of a smooth function. Thus, $b(p):=\int_{F(u)\cap \overline{B_r(p)}} H \, d\H^1<\infty$. Since $F(u)\cap \overline{B_{2\rho}}$ is compact, there are finitely many disks $\{B_{r_k}(p_k)\}_{k=1}^N$ with $p_k\in  F(u)\cap \overline{B_{2\rho}}$ that cover it, implying
\[
\int_{F(u)\cap \overline{B_{2\rho}}} H \, d\H^{1} \leq \sum_{k=1}^N \int_{F(u)\cap B_{r_k}(p_k)} H \, d\H^{1} = \sum_{k=1}^N b(p_k) < \infty.
\]
Adding this estimate to \eqref{lem:curvbd:eq}, we get that the total mean curvature $\int_{F(u)} H \, d\H^1<\infty$.

\end{proof}

We are now in a position to demonstrate that a finite-index solution in the plane has a free boundary, consisting of finitely many connected components. 

\begin{prop}\label{prop:fincurvtop}
 Suppose that $u:\R^2 \to [0,\infty)$  is a global classical solution of \eqref{FBP} with a connected positive phase $\O=\{u>0\}$ and finite index. Then $u$ has finite topology, that is, $F(u)$ consists of finitely many connected components. 
\end{prop}
\begin{proof}

First, we assert that there are only finitely many connected components $\a$ of $F(u)$ such that $\a \cap \overline{B_R}\neq \emptyset$. Indeed, if there were infinitely many such $\{\a_k\}_{k=1}^\infty$ then by compactness, after possibly relabeling the subsequence, we could find points $p_k\in \a_k$ such that $p_k\to p\in \overline{B_{R}}$. But since $F(u)$ is closed, it must be that $p_\infty\in F(u)$. However, since $u$ is a classical solution of \eqref{FBP}, there exists a small enough $r>0$ such that $F(u)\cap B_r(p)$ is the graph of a smooth function over the tangent line to $F(u)$ at $p$. In particular, there is only one connected component of $F(u)$ intersecting $B_r(p)$, which contradicts the fact that the approaching sequence $p_k\to p$ came from distinct connected component of $F(u)$. 

Fix $0<\e<\pi/2$ small. By Lemma \ref{lem:meancurvest}, there exists an $R>0$ large enough such that 
\begin{equation}\label{prop:fincurvtop:curv}
\int_{F(u)\setminus \overline{B_R}} H d \H^1 < \e.
\end{equation}
We will show that there can be at most two connected components $\a$ of $F(u)$ such that $\a \subset \R^2\setminus \overline{B_R}$, which will complete the proof of the proposition.  For the purpose, we point out that if $\a=\a(s)$ is parametrized by arclength, such that $(\a',\nu)$ is a positive basis of $\R^2$, then the mean curvature $H$ of $\a$ with respect to the outer unit normal $\nu$ to $\a\subset \de \O$,
\begin{equation}
0\leq H(\a(s)) = \a''(s)\cdot \nu = |\a''(s)| =\k(s)
\end{equation}
coincides with the curvature $\k(s)$ of $\a$. Hence, its integral
\begin{equation}\label{prop:fincurvtop:turn}
\int_{\a([s_1,s_2])} H d\H^{1} = \int_{s_1}^{s_2} \k(s) ds
\end{equation}
gives the angle of turning of the unit tanget $\a'(s)$ as we trace the curve from $s_1$ to $s_2$. In particular, the smooth submanifold $\a$ cannot be diffeomorphic to a circle, since we would then have $\int_\a H \, d\H^1 = 2\pi > \e$, contradicting \eqref{prop:fincurvtop:curv}. Thus, each of these $\a$ is diffeomorphic to $\R$.

Pick a connected component $\a$ of $F(u)$ that is contained in $\R^2\setminus \overline{B_R}$, and denote by $Z_\a$ the connected component of $\{u=0\}$ bordering $\a$. Since $H\geq 0$, $Z_\a$ is a closed convex subset of $\R^2$. Let $p\in \a$. Claim there exists an open sector with a vertex at $p$,
$$
S_\a=\{x\in \R^2: x= p+ s v_1 + t v_2, ~ s,t > 0\}  \quad \text{for some unit vectors } v_1, v_2\subset \R^2,
$$ 
such that $\angle(v_1, v_2)= \pi -2\e$ and $S_\a\subset Z_\a$. By applying a rigid motion, we may assume that $p=0$, the tangent line to $F(u)$ at $p$ is the $x_1$-axis and that the convex $Z_\a$ is contained in the half-plane $\{x_2\geq 0\}$. Take $v_1$ and $v_2$ to make angles $\e$ and $\pi-\e$ with the $x_1$-axis, respectively, and denote  by $\mathcal{R}_1, \mathcal{R}_2$ the open rays generated by them. Parametrize $\a$ by arclength such that $\a(0)=0$ and $\a'(0)=e_1$. Note that because of the mean curvature bound \eqref{prop:fincurvtop:curv}, the angle of turning of $\a'$ as we trace the curve from $0=\a(0)$ to $q=\a(s)$, $s\geq 0$, 
\begin{equation}\label{prop:fincurvtop:theta}
\theta(q) := \int_{\a([0,s])} H\,d\H^1 < \e <\pi/2,
\end{equation}
so that $\a'(s)\cdot e_1 = \cos \theta(q) \geq 0$ whenever $s\geq 0$. In particular, $\a(s)$ is inside the closed first quadrant of the plane for all $s\geq 0$.  Similarly, $\a(s)$ belongs to the closed second quadrant for all $s\leq 0$.

Assume that the sector $S_\a\not\subset Z_\a$.  Then either $\mathcal{R}_1$ or $\mathcal{R}_2$ intersects $\a$: without loss of generality, we may assume that it is $\mathcal{R}_1$. Since $\a(s)$ belongs to the closed second quadrant for $s\leq 0$, whereas the ray $\mathcal{R}_1$ is contained in the open first quadrant, there exists a smallest positive value $s_*$ such that $\a(s_*) \in \mathcal{R}_1\cap \a$. However, in order to intersect $\mathcal{R}_1$ at $q=\a(s_*)$, the curve $\a$ needs to have a tangent vector $\a'(s_*)$ at $q$ that makes an angle $\geq \e$ with the $x_1$-axis. This contradicts \eqref{prop:fincurvtop:theta}.

Now we can complete the proof after observing that for $\e>0$ small enough, there cannot be more than two disjoint sectors $S_{\a_i}\subseteq Z_{\a_i}$, $i=1,2$, of opening $(\pi-2\e)$. 

\end{proof}

\begin{proof}[Proof of main classification Theorem \ref{thm:class}]
According to Proposition \ref{prop:fincurvtop}, $F(u)$ has finitely many connected components. By Traizet's Theorem \cite[Theorem 12]{Traizet},  $u$ is either the one-plane solution, or the disk-complement solution, or the double hairpin solution, up to a similarity transformation. In the first case, $H\equiv 0$ and $u$ is obviously stable. The computation that the Morse index is $1$ in the remaining two cases is carried out in the next section: in Proposition \ref{prop:relind} and Lemma \ref{lem:calcindex}.
\end{proof}

\section{Computations of the Morse index}\label{sec:Morsecomp}

Let $u:\R^2\to [0,\infty)$ denote either the disk-complement solution $L(x):=(\log|x|)^+$ or the double hairpin solution  $H(x)$ of \cite{HHP}, and $\O:=\{u>0\}$ be the positive phase of $u$. We will identify $\R^2\simeq \C$, where the standard complex variable $z:=x_1+ix_2$. Since $u$ is harmonic and smooth up to the free boundary $F(u)=\de \O$,
\begin{equation}
G(z) := 2 \de_z u = u_{x_1} - i u_{x_2}, \quad z\in \O,
\end{equation}
is a holomorphic function in $\O$ that extends smoothly to $\de \O$, and satisfies 
\begin{equation}
|G(z)|= 1 \quad \text{whenever } z\in \de \O. 
\end{equation}
In both cases $G:\O\to \Dd$ maps $\O$ biholomorphically onto its image $\Sigma:=G(\O)$:
\begin{itemize}
\item when $u=L$, $\Sigma = \Dd\setminus \{0\}$ and $G(z)= 1/z$ extends to a smooth bijection from $\overline{\O} = \{|z|\leq 1\}$ to $\tilde{\Sigma}:=\overline{\Dd}\setminus \{0\}$, mapping $F(u)$ bijectively onto $\de \Dd$; \\
\item when $u=H$, $\Sigma = \Dd$ and $G$ extends to a smooth bijection from $\overline{\O}$ to $\tilde{\Sigma}:=\overline{\Dd}\setminus \{\pm 1\}$, mapping $F(u)$ bijectively onto $\de \Dd \setminus \{\pm 1\}$ (see \cite[Section 4]{HHP}).\\
\end{itemize}

Let $(\xi_1, \xi_2)$ denote the standard Euclidean coordinates on $\Dd$ and $\z:=\xi_1 + i \xi_2$ -- the corresponding complex coordinate. We will use the coordinates $\zeta\in \Sigma$  to parametrize $\O$ via $\zeta \to z := G^{-1}(\zeta)$.

\begin{lemma}\label{lem:coordchange} Let $u=L$ or $u=H$, and $\O=\{(x_1, x_2)\in \R^2:u(x_1,x_2)>0\}$. Then 
\begin{equation}\label{lem:coordchange:eq}
\begin{split}
Q(\phi, \phi) &=\int_{\O} |\n \phi|^2 \, dx - \int_{\de \O} H \phi^2 \, d\H^1 \\
&= \int_{\Dd} |\nabla \psi|^2 \, d\xi - \int_{\de \Dd} \psi^2\, d\H^1, \quad \text{for all } \phi \in C^\infty_c(\overline{\O}) \text{ and } \psi=\phi\circ G^{-1} \in C^\infty_c(\tilde{\Sigma}).
\end{split}
\end{equation}
\end{lemma}
\begin{proof}
By the conformal invariance of the Dirichlet energy, we have
\begin{equation}
\int_{\O} |\n \phi|^2 \, dx =  \int_{\Dd} |\nabla \psi|^2 \, d\xi, \quad \text{for } \phi \in C^\infty_c(\overline{\O}) \text{ and } \psi = \phi\circ G^{-1}.
\end{equation}
To get the equality of the boundary integral terms in \eqref{lem:coordchange:eq}, we first note that for any $z\in \de\O$,
\begin{equation}\label{lem:coordchange:H1}
H(z) = \text{div}\left(\frac{\nabla u}{|\nabla u|}\right) = - \frac{\n u \cdot \n |\n u|^2}{2 |\n u|^3} = - \text{Re}(\overline{G(z)} \de_{z} |G(z)|^2) = \text{Re}(-\overline{G(z)}^2 G'(z) ),
\end{equation}
as $|\n u|^2 =1$ on $\de \O$ and $G(z) = u_{x_1}- i u_{x_2}$. Furthermore, since $G:\overline{\O}\to \tilde{\Sigma}$ is a conformal map, its (real-valued) Jacobian $G_*$ maps the outer unit normal vector $\nu\in T_z \overline{\O}$, $z\in \de \O$, to the the exterior unit vector $v$ to $\de \Dd$ at $\zeta = G(z)$, multiplied by the conformal factor $|G'(z)|$:
\begin{equation}\label{lem:coordchange:pushforward}	
G_* (\nu) = |G'(z)| v,
\end{equation}
Using the standard identification of real 2-vectors with complex numbers, we have 
\begin{align*}
\nu = -\n u  &\longleftrightarrow  - u_{x_1} - i u_{x_2} = - \overline{G(z)} \\
v  &\longleftrightarrow  \zeta = G(z),
\end{align*}
so that we can rewrite \eqref{lem:coordchange:pushforward} as
\begin{equation}\label{lem:coordchange:H2}
G'(z) (-\overline{G(z)}) = G_* (\nu) = |G'(z)| v= |G'(z)| G(z) \quad \text{for all } z\in \de\O.
\end{equation}
Using \eqref{lem:coordchange:H1},  \eqref{lem:coordchange:H2} and the transformation of length under a conformal map $$|d\zeta| = |G'(z)| |dz|,$$ we obtain that 
\begin{align*}
 H(z) |dz| = \frac{\text{Re}(-\overline{G(z)}^2 G'(z) )}{|G'(z)|} |d\zeta| = \text{Re}(\overline{G(z)} G(z)) |d\zeta| = |d\zeta| \quad \text{on } \de \O,
\end{align*}
since $|G(z)|=1$ on $\de \O$. Hence, after changing variables $(x_1,x_2)\to (\xi_1, \xi_2)$ in the integration, we get the desired equality:
\[
\int_{\de \O} H \phi^2 \, d\H^1 =  \int_{\de \Dd} \psi^2\, d\H^1.
\]
\end{proof}

In view of Lemma \ref{lem:coordchange}, define the quadratic form
\begin{equation}\label{eq:Q0}
Q_0(\psi, \psi):= \int_\Dd |\n \psi|^2 \, d\xi - \int_{\de \Dd} (T\psi)^2 d\H^1 \quad \text{for } \psi \in H^1(\Dd),
\end{equation}
where $T:H^1(\Dd)\to L^2(\de\Dd)$ denotes the trace operator, and denote by $\ind(Q_0)$ the index of $Q_0$ in $H^1(\Dd)$. The associated eigenvalue problem to \eqref{eq:Q0} reads:
\begin{equation}\label{eq:Q0eig}
\left\{
\begin{aligned}
   -\Delta  \psi = \l \psi & \quad \mbox{in } \Dd, \\
   \psi_\nu - \psi = 0 & \quad \mbox{on }  \de\Dd. \\
  \end{aligned}\right.
\end{equation}
Arguing as in Theorem \ref{thm:std}, we see that the Robin problem \eqref{eq:Q0eig} possesses a discrete set of real eigenvalues $\l_k \to \infty$ and a set of corresponding eigenfunctions $\{\psi_k\}_{k\in \N}\subset H^1(\Dd)\cap C^\infty(\overline{\Dd})$, which form an orthonormal basis for $L^2(\Dd)$. Furthermore, $\ind(Q_0)$ is precisely the number of negative eigenvalues of \eqref{eq:Q0eig}. The next proposition relates the index of the disk-complement solution and the double hairpin solution to $\ind(Q_0)$.

\begin{prop}\label{prop:relind} Let $u = L$ or $u = H$. Then $\ind(u)$ equals the index of $Q_0(\cdot, \cdot)$ in $H^1(\Dd)$.
\end{prop}
\begin{proof}
Let $R>0$ and consider the bounded domain $\O_R:=\O\cap B_R$. Let $V\subset \X_{\O_R}$ be a subspace of dimension $\ind(u, \O_R)$, on which $Q(\cdot, \cdot; \O_R)$ is negative definite. Extending each $\phi \in V$ trivially to a function defined on all $\O$, we clearly have $\psi:=\phi\circ G^{-1}\in C^\infty_c(\tilde{\Sigma}) \subset H^1(\Dd)$ since $\supp \phi$ is compact and $G$ is smooth on $\overline{\O}$. Lemma \ref{lem:coordchange} then tells us that
\[
0>Q(\phi,\phi; \O_R)= Q_0(\psi,\psi) \quad \text{for all } \phi \in V.
\]
Hence $\ind(Q_0)\geq \dim\, V = \ind(u, \O_R)$, and taking $R\to \infty$, we obtain $\ind(u)\leq \ind(Q_0).$

To get the opposite inequality, let $k=\ind(Q_0)$ and let $\{\psi_i\}_{i=1}^k \subset C^\infty(\overline{\Dd})$ be an $[-Q_0(\cdot, \cdot)]$-orthonormal set of eigenfunctions for \eqref{eq:Q0eig}, associated to the negative eigenvalues. Denote by $P$ the singular point set of $G^{-1}$ in $\overline{\Dd}$, i.e. $P=\{0\}$ in the case of $u=L$, and $P=\{\pm 1\}$ in the case of $u=H$. Using the logarithmic cut-off function
\[
\eta_\e(\xi) = \begin{cases}0 & |\xi|<\e^2, \\ \frac{\log (|\xi|/\e^2)}{\log (1/\e)} & \e^2\leq |\xi| < \e, \\ 1 & |\xi|\geq \e, \end{cases}
\] 
we can construct approximations $\psi_i^\e(\xi):= \psi_i(\xi) \prod_{p\in P} \eta_\e(\xi - p)$ of the eigenfunctions $\psi_i(\xi)$, $i=1,\ldots,k$, that satisfy:
\begin{enumerate}
\item $\psi_i^\e \in H^1(\Dd)$ with support in $\overline{\Dd}\setminus \bigcup_{p\in P} \overline{B_{\e^2}(p)}$;
\item $\lim_{\e\to 0}\| \psi_i - \psi_i^\e\|_{H^1(\Dd)} = 0. $ 
\end{enumerate}
Define $W^\e:=\text{Span}(\{\psi_i^\e\}_{i=1}^k)$. Because of (2) and the $L^2$-trace inequality, we have 
\begin{equation}\label{prop:relind:contofQ0}
\lim_{\e\to 0} Q_0(\psi_i^\e, \psi_j^\e) = Q_0(\psi_i, \psi_j) = -\d_{ij}, \qquad i,j=1,\ldots, k,
\end{equation}
so that for all small $\e>0$, $Q_0$ is negative definite on $W^\e$, whose $\dim\, W^\e = k$. Furthermore, (1) implies that for every $\e>0$ there exists $R=R(\e)$, such that $\phi^\e:=\psi^\e\circ G$ has compact support in $\overline{\O}\cap B_R$ for every $\psi^\e\in W^\e$. Hence, $V^\e:=\{\psi^\e\circ G\lfloor_{\O_R}: \psi^\e \in W^\e\}\subset H^1_F(\O_{R(\e)})$ and Lemma \ref{lem:coordchange} entails that $Q(\cdot, \cdot, \O_{R(\e)})$ is negative definite on $V^\e$. We conclude that for all small $\e>0$,
\[
\ind(Q_0)=k = \dim\, V^\e \leq \ind(u, \O_{R(\e)}) \leq \ind(u). 
\]
\end{proof}

In the last lemma below, we compute the index of $Q_0$ and complete the proof of Theorem \ref{thm:class}.

\begin{lemma}\label{lem:calcindex} The index of $Q_0$ in $H^1(\Dd)$ is $1$. 
\end{lemma}
\begin{proof} First, note that if $a_\psi:=\frac{1}{2\pi}\int_{\de \Dd}\psi$ denotes the average value of the trace of $\psi \in H^1(\Dd)$ on $\de \Dd$, then
\begin{equation}\label{lem:calcindex:av}
Q_0(\psi, \psi) = \int_{\Dd} |\n (\psi - a_\psi)|^2 \, dx - \int_{\de \Dd} (\psi - a_\psi + a_\psi)^2 \, d\H^1 = Q_0(\psi-a_\psi,\psi-a_\psi) - 2\pi a_\psi^2. 
\end{equation}
In particular, we see that when $\psi\equiv 1$, the form $Q_0(\psi,\psi)<0$, so that $\ind(Q_0)\geq 1$. On the other hand, since $\psi \to a_\psi$ is a bounded linear functional on $H^1(\Dd)$, its nullspace
\[
W:=\{\psi\in H^1(\Dd): a_\psi = 0\},
\]
is codimension $1$ in $H^1(\Dd)$. Therefore, \eqref{lem:calcindex:av} will imply that $\ind(Q_0)=1$ once we show that 
\begin{equation*}\label{lem:calcindex:main}
Q_0(\psi, \psi) \geq 0 \quad \text{for all } \psi \in W. 
\end{equation*}
By the Dirichlet principle, it suffices to prove that
\begin{equation}\label{lem:calcindex:har}
Q_0(h,h)=\int_{\Dd} |\n h|^2 \, dx - \int_{\de \Dd} h^2 \, d\H^1 \geq 0 \quad \text{for all harmonic functions } h\in W. 
\end{equation}
Since $h\in W\subset H^1(\Dd)$, its zero-mean trace on $\de \Dd$ is in $H^{1/2}(\Dd)$, meaning that the latter has a Fourier series expansion in $\theta \in \de\Dd$:
\begin{equation}\label{lem:calcindex:fourier}
h(\theta) = \sum_{k=1}^\infty \left(c_k e_k(\theta) + d_k f_k(\theta)\right), \quad \text{with} \quad \sum_{k=1}^\infty k (c_k^2 +d_k^2) <\infty,
\end{equation}
where $e_k := (\cos \k \theta)/\sqrt{\pi}$ and $f_k(\theta):=(\sin \k \theta)/\sqrt{\pi}$, $k\in \N$, satisfy:
\begin{equation*}
\langle e_k, e_l \rangle_{L^2(\de \Dd)} = \d_{kl} = \langle  f_k, f_l \rangle_{L^2(\de \Dd)},  \text{ and } \langle e_k, f_l \rangle_{L^2(\de \Dd)} = 0.
\end{equation*}
Now, the harmonic $h$ inside the disk can be written as the series
\begin{equation*}
h(r,\theta) = \sum_{k=1}^{\infty} r^k\left(c_k e_k(\theta) + d_k f_k (\theta)\right),
\end{equation*}
which is convergent in $H^1(\Dd)$ on account of \eqref{lem:calcindex:fourier} and the fact that for all $k,l\in \N$,
\begin{align*}
\langle r^k e_k, r^l e_l \rangle_{L^2(\Dd)} = \d_{kl}/(2k+1) = \langle r^k f_k, r^l f_l \rangle_{L^2(\Dd)} & \text{ and } \langle r^k e_k, r^l f_l \rangle_{L^2(\Dd)}=0, \\
\langle \n (r^k e_k), \n (r^l e_l) \rangle_{L^2(\Dd)} = k \d_{kl} = \langle \n (r^k f_k), \n (r^l f_l) \rangle_{L^2(\Dd)} & \text{ and } \langle \n (r^k e_k), \n (r^l f_l) \rangle_{L^2(\Dd)}=0.
\end{align*}
Therefore, we get
\begin{align*}
Q_0(h,h) = \sum_{k=1}^\infty k (c_k^2 + d_k^2)  - \sum_{k=1}^\infty (c_k^2 + d_k^2) = \sum_{k=1}^\infty (k-1) (c_k^2 + d_k^2) \geq 0,
\end{align*}
which is the desired \eqref{lem:calcindex:har}.

\end{proof}

\appendix

\section{Sobolev trace inequalities}
\begin{prop}\label{prop:trace} Let $u$ be a classical solution of the one-phase FBP in a bounded domain $D\subset \R^n$ and let $\O$ be a connected component of the positive phase $D^+(u)$. Assume that 
\begin{equation}\label{prop:trace:lip}
|\n u| \leq L  \quad \text{in } \O.
\end{equation}
Then 
\begin{equation}\label{prop:trace:L1}
\int_{F(u)\cap \de \O} \psi \, dx \leq L \|\n \psi \|_{L^1(\O)} \quad \text{for all } \psi \in C^1_c(\O\cup (F(u)\cap \de \O)).
\end{equation}
As a result, the following Sobolev $L^2$-trace inequality holds:
\begin{equation}\label{prop:trace:ineq1}
\|T \psi\|^2_{L^2(F(u)\cap \de \O)}\leq 2L \|\psi\|_{L^2(\O)} \|\n \psi\|_{L^2(\O)} \quad \text{for all } \psi \in H^1_F(\O). 
\end{equation}
In particular, we have for any $\eps>0$: 
\begin{equation}\label{prop:trace:ineq2}
\|T \psi\|^2_{L^2(F(u)\cap \de \O)}\leq \eps \|\n \psi\|_{L^2(\O)}^2 + (L^2/\eps) \|\psi\|_{L^2(\O)}^2  \quad \text{for all } \psi \in H^1_F(\O). 
\end{equation}
\end{prop}
\begin{proof}
Once we establish $\eqref{prop:trace:L1}$, the trace inequality \eqref{prop:trace:ineq1} follows by approximating $\psi\in H^1_F(\O)$ with functions in $C^1_c(\O\cup (F(u)\cap \de \O)$, plugging in $\psi^2$ in \eqref{prop:trace:L1} and using the Cauchy-Schwarz inequality. In turn, \eqref{prop:trace:ineq2} results from applying to \eqref{prop:trace:L1} the inequality $2 ab \leq \eps a^2 + \eps^{-1} b^2$, with $a=\|\n \psi\|_{L^2(\O)}$ and $b=L \|\psi\|_{L^2(\O)}$. 

In order to prove \eqref{prop:trace:L1}, we use the fact that $u\in C^\infty(\overline{\O})$ is harmonic in $\O$ with $u_\nu = -1$ on $F(u)$, and apply the Divergence theorem to the vector field $\psi \n u$, where $\psi \in C^1_c(\O\cup (F(u)\cap \de \O))$:
\begin{align*}
\int_{\de \O\cap F(u)} \psi \, d\H^{n-1} &= \int_{\de \O} \psi (-u_\nu) \, d\H^{n-1} = - \int_\O \div(\psi \n u)\, dx \\
& = - \int_{\O} \n \psi\cdot \n u \,dx - \int_\O \psi \D u \, dx \\
& \leq \int_\O |\n \psi| |\n u| \, dx  \leq L \|\n \psi\|_{L^1(\O)}.
\end{align*}
where the last inequality follows from the Lipschitz bound \eqref{prop:trace:lip}.
\end{proof}

\bibliography{morse_bib}

\end{document}